\documentclass[12pt,reqno]{amsart}
\usepackage{amsmath,amssymb,amsfonts,amscd,latexsym,amsthm,mathrsfs}
\usepackage[usenames]{color}
\usepackage[unicode]{hyperref}
\makeatletter
\renewcommand\thesubsection{\@arabic\c@subsection}
\makeatother

\let\wh\widehat
\let\wt\widetilde
\renewcommand{\d}{{\mathrm d}}
\newcommand{\m}{{\mathrm m}}

\renewcommand{\Re}{\operatorname{Re}}

\newtheorem{theorem}{Theorem}
\newtheorem{lemma}{Lemma}

\theoremstyle{remark}

\begin{document}

\title{Short walk adventures}

\date{18 January 2018. \emph{Revised}: 29 April 2018}

\author{Armin Straub}
\address{Department of Mathematics and Statistics, University of South Alabama, 411 University Blvd N, MSPB 325, Mobile, AL 36688, USA}
\email{straub@southalabama.edu}

\author{Wadim Zudilin}
\address{Department of Mathematics, IMAPP, Radboud University, PO Box 9010, 6500~GL Nijmegen, Netherlands}
\email{w.zudilin@math.ru.nl}

\address{School of Mathematical and Physical Sciences, The University of Newcastle, Callag\-han, NSW 2308, Australia}
\email{wadim.zudilin@newcastle.edu.au}

\address{Laboratory of Mirror Symmetry and Automorphic Forms, National Research University Higher School of Economics, 6 Usacheva str., 119048 Moscow, Russia}
\email{wzudilin@gmail.com}

\dedicatory{To the memory of Jon Borwein, who convinced us that a short walk can be adventurous}

\thanks{Both authors acknowledge support of the Max Planck Society during their stays at the
Max Planck Institute for Mathematics (Bonn, Germany) in the years 2015--2017.
The second author is partially supported by Laboratory of Mirror Symmetry NRU HSE, RF government grant, ag.\ no.\ 14.641.31.0001.}



\begin{abstract}
We review recent development of short uniform random walks, with a focus on its connection to (zeta) Mahler measures and
modular parametrisation of the density functions. Furthermore, we extend available ``probabilistic'' techniques
to cover a variation of random walks and reduce some three-variable Mahler measures, which are conjectured to evaluate in terms of $L$-values of modular forms, to hypergeometric form.
\end{abstract}

\maketitle

\setcounter{section}{-1}
\section{Introduction}
\label{s0}

At some stages of our careers we were approached by Jon Borwein to collaborate on a theme that sounded rather
off topic to us, who had interests in number theory, combinatorics and related special functions. Somewhat unexpectedly, the theme has
become a remarkable research project with several outcomes (including \cite{BNSW11,BSW13,BSWZ12}, to list a few),
a project which we continue to enjoy after the sudden loss of Jon\dots\
This note serves as a summary to our recent discoveries that certain ``probabilistic'' techniques apply usefully to tackling
difficult problems on the border of analysis, number theory and differential equations; in particular, in evaluating
multi-variable Mahler measures. Our principal novelties are given in Theorems \ref{th1}--\ref{th3}; these include
hypergeometric reduction of the Mahler measures of the three-variable polynomials
$$
1+x_1+x_2+x_3+x_2x_3 \quad\text{and}\quad (1+x_1)^2+x_2+x_3,
$$
as well as the (hypergeometric) factorisation of a related differential operator for the Ap\'ery-like sequence
$$
\sum_{k=0}^n{\binom nk}^2{\binom{2k}k}^2,
\quad\text{where}\; n=0,1,2,\dotsc.
$$

Echoing Jon's ``a short walk can be beautiful'' \cite{Bo16}, we add that ``a short walk can be adventurous.''

\section{Uniform random walks}
\label{s1}

An $N$-step uniform random walk is a planar walk that starts at the origin and
consists of $N$ steps of length $1$ each taken into a uniformly random direction.
Let $X_N$ be the distance to the origin after these $N$ steps.
The $s$-th moments $W_N(s)$ of $X_N$ can be computed \cite{BSWZ12} via the formula
\begin{align*}
W_N(s)
&=\idotsint_{[0,1]^N}|e^{2\pi i\theta_1}+\dots+e^{2\pi i\theta_N}|^s\,\d\theta_1\dotsb\d\theta_N
\\
&=\idotsint_{[0,1]^{N-1}}|1+e^{2\pi i\theta_1}+\dots+e^{2\pi i\theta_{N-1}}|^s\,\d\theta_1\dotsb\d\theta_{N-1},
\end{align*}
and are related to the (probability) density function $p_N(x)$ of $X_N$ via
$$
W_N(s)=\int_0^\infty x^sp_N(x)\,\d x
=\int_0^Nx^sp_N(x)\,\d x.
$$
That is, $p_N(x)$ can then be obtained as the inverse Mellin transform of $W_N(s-1)$.
Finally, note that the even moments $W_3(2n)$ and $W_4(2n)$ (which are, clearly, positive integers) can be identified with the odd moments of
$I_0(t)K_0(t)^2$ and $I_0(t)K_0(t)^3$, respectively,
where $I_0(t)$ and $K_0(t)$ denote the modified Bessel functions of the first and second kind. Namely, for $n=1,2,\dots$ we have \cite{BBBG08}
\begin{equation*}
W_3(2n)=\frac{3^{2n+3/2}}{\pi\,2^{2n}\,n!^2}\int_0^\infty t^{2n+1}I_0(t)K_0(t)^2\,\d t
\end{equation*}
and
\begin{equation*}
W_4(2n)=\frac{4^{2n+2}}{\pi^2\,n!^2}\int_0^\infty t^{2n+1}I_0(t)K_0(t)^3\,\d t.
\end{equation*}

\section{Zeta Mahler measures}
\label{s2}

For a non-zero Laurent polynomial $P(x_1,\dots,x_N)\in\mathbb C[x_1^{\pm1},\dots,x_N^{\pm1}]$, its zeta Mahler measure \cite{Ak09} is defined by
$$
Z(P;s)=\idotsint_{[0,1]^N}|P(e^{2\pi i\theta_1},\dots,e^{2\pi i\theta_N})|^s\,\d\theta_1\dotsb\d\theta_N,
$$
and its logarithmic Mahler measure is
$$
\m(P)=\frac{\d Z(P;s)}{\d s}\bigg|_{s=0}=\idotsint_{[0,1]^N}\log|P(e^{2\pi i\theta_1},\dots,e^{2\pi i\theta_N})|\,\d\theta_1\dotsb\d\theta_N.
$$
A straightforward comparison of the two definitions reveals that
$$
W_N(s)=Z(x_1+\dots+x_N;s)=Z(1+x_1+\dots+x_{N-1};s)
$$
and
\begin{equation}
W_N'(0)=\m(x_1+\dots+x_N)=\m(1+x_1+\dots+x_{N-1})=\int_0^Np_N(x)\log x\,\d x,
\label{red0}
\end{equation}
where the derivative is with respect to $s$. The latter Mahler measures are known as linear Mahler measures.
The evaluations
\begin{align*}
W_2'(0)=0, \quad W_3'(0)=L'(\chi_{-3};-1)=\frac{3\sqrt3}{4\pi}\,L(\chi_{-3};2), \quad W_4'(0)=-14\zeta'(-2)=\frac{7\zeta(3)}{2\pi^2},
\end{align*}
are known \cite{Sm81}, while the following conjectural evaluations, due to Rodriguez-Villegas \cite{BLRVD03} and verified to several hundred digits \cite{BB12}, remain open:
\begin{align*}
W_5'(0)&\overset?=-L'(f_3;-1)=6\biggl(\frac{\sqrt{15}}{2\pi}\biggr)^5L(f_3;4),
\\
W_6'(0)&\overset?=-8L'(f_4;-1)=3\biggl(\frac{\sqrt6}{\pi}\biggr)^6L(f_4;5),
\end{align*}
where
$$
f_3(\tau)=\eta(\tau)^3\eta(15\tau)^3+\eta(3\tau)^3\eta(5\tau)^3
\quad\text{and}\quad
f_4(\tau)=\eta(\tau)^2\eta(2\tau)^2\eta(3\tau)^2\eta(6\tau)^2
$$
are cusp eigenforms of weight 3 and 4, respectively.
Here and in what follows, Dedekind's eta function
$$
\eta(\tau)=q^{1/24}\prod_{m=1}^{\infty}(1-q^m)
=\sum_{n=-\infty}^{\infty}(-1)^nq^{(6n+1)^2/24},
\quad\text{where}\; q=e^{2\pi i\tau},
$$
serves as a principal constructor of modular forms and functions.
No similar formulae are known for $W_N'(0)$ when $N\ge7$, though the story continues at a different level\,---\,see
\cite{Br16,Zh18,Zh18b} for details.

\section{Generic two-step random walks}
\label{s2+}

Let $X_1$ and $X_2$ be two (sufficiently nice, independent) random variables on $[0,\infty)$
with probability density $p_1(x)$ and $p_2(x)$, respectively, and
let $\theta_1$ and $\theta_2$ be uniformly distributed on $[0, 1]$.
Then $X = e^{2\pi i\theta_1}X_1 + e^{2\pi i\theta_2}X_2$ describes a two-step random walk
in the plane with a first step of length~$X_1$ and a second step of length~$X_2$.
As in \cite[eq.~(3-3)]{BSW13}, an application of the cosine rule shows that the $s$-th moment of $|X|$ is
\begin{equation*}
W(s) = \mathsf{E}(|X|^s) = \int_0^{\infty}\!\!\!\int_0^{\infty} g_s (x, y) p_1 (x) p_2 (y)\,\d x\,\d y,
\end{equation*}
where
\begin{equation*}
g_s (x, y) = \frac{1}{\pi} \int_0^{\pi} (x^2 + y^2 + 2 x y \cos \theta)^{s/2}\,\d\theta.
\end{equation*}
Observe that
\begin{equation*}
\frac{\d g_s (x, y)}{\d s}\bigg|_{s = 0}
= \frac{1}{\pi} \int_0^{\pi} \log {\textstyle\sqrt{x^2 + y^2 + 2 x y \cos \theta}} \,\d\theta
= \max\{\log |x|, \log |y|\},
\end{equation*}
so that, in particular,

\begin{lemma}
\label{lem1}
We have
\begin{equation*}
W'(0) = \mathsf{E} (\log |X|)
= \int_0^{\infty}\!\!\!\int_0^{\infty} p_1 (x) p_2 (y) \max \{\log x, \log y\} \,\d y\,\d x.
\end{equation*}
\end{lemma}

Alternative equivalent expressions, that will be useful in what follows, include
\begin{align}
\mathsf{E} (\log |X|)
& = \int_0^{\infty}\!\!\!\int_0^x p_1 (x) p_2 (y) \log x \,\d y\,\d x
+ \int_0^{\infty}\!\!\!\int_x^{\infty} p_1 (x) p_2 (y) \log y \,\d y\,\d x
\nonumber\\
& = \mathsf{E} (\log X_1) + \int_0^{\infty}\!\!\!\int_x^{\infty} p_1 (x) p_2 (y)  (\log y - \log x) \,\d y\,\d x
\nonumber\\\label{eq:Wp0:p1p2}
& = \mathsf{E} (\log X_2) + \int_0^{\infty}\!\!\!\int_0^x p_1 (x) p_2 (y) (\log x - \log y) \,\d y\,\d x.
\end{align}

\section{Linear Mahler measures}
\label{s2a}

Let $N,M$ be integers such that $N>M>0$.
By decomposing an $N$-step random walk into two walks with $N-M$ and $M$ steps, and
applying Lemma~\ref{lem1} in the form \eqref{eq:Wp0:p1p2}, we find that
\begin{equation*}
W_N' (0) = W_M' (0) + \int_0^{N - M} p_{N - M} (x) \biggl( \int_0^x p_M (y) (\log x - \log y) \,\d y \biggr) \,\d x.
\end{equation*}
This formula, together with known formulae for the densities \cite{BSWZ12}, like $p_1(x)=\delta(x-1)$ (the Dirac delta function) and
$p_2(x)=2/(\pi\sqrt{4-x^2})$ for $0<x<2$, allows one to produce new expressions for linear Mahler measures. Indeed, taking $M=1$ we get
\begin{equation}
W_N'(0)
=\int_1^{N-1}p_{N-1}(x)\log x\,\d x
\label{red1}
\end{equation}
(which can be also derived using Jensen's formula), while $M=2$ results in
\begin{equation}
W_N'(0)
=\int_2^{N-2}p_{N-2}(x)\log x\,\d x
+\frac1\pi\int_0^2p_{N-2}(x)x\cdot{}_3F_2\biggl(\begin{matrix} \tfrac12, \, \tfrac12, \, \tfrac12 \\ \tfrac32, \, \tfrac32 \end{matrix}\biggm| \frac{x^2}{4} \biggr)\,\d x
\label{red2}
\end{equation}
(see also \cite[eq.~(2.1)]{Ro06}).
Here, and in what follows, the hypergeometric notation
$$
{}_mF_{m-1}\biggl(\begin{matrix}
a_1, \, a_2, \, \dots, \, a_m \\
\phantom{a_1,} \, b_2, \, \dots, \, b_m \end{matrix} \biggm|z\biggr)
=\sum_{n=0}^\infty
\frac{(a_1)_n(a_2)_n\dotsb(a_m)_n}
{(b_2)_n\dotsb(b_m)_n}\,
\frac{z^n}{n!}
$$
is used, where
$$
(a)_n=\frac{\Gamma(a+n)}{\Gamma(a)}=\begin{cases}
a(a+1)\dotsb(a+n-1), &\text{for $n\ge1$}, \\
1, &\text{for $n=0$},
\end{cases}
$$
denotes the Pochhammer symbol (the rising factorial).
Note that we deduce \eqref{red2} from
$$
\int_0^x p_2(y) (\log x - \log y)\,\d y
= \frac{x}\pi\cdot{}_3F_2\biggl(\begin{matrix} \tfrac12, \, \tfrac12, \, \tfrac12 \\ \tfrac32, \, \tfrac32 \end{matrix}\biggm| \frac{x^2}{4} \biggr),
$$
which is valid if $0\le x\le 2$.

Equations \eqref{red1} and \eqref{red2} and the formula
$$
p_4(x)=\frac{2\sqrt{16-x^2}}{\pi^2x}\,
\Re{}_3F_2\biggl(\begin{matrix} \tfrac12, \, \tfrac12, \, \tfrac12 \\ \tfrac56, \, \tfrac76 \end{matrix}\biggm| \frac{(16-x^2)^3}{108x^4} \biggr)
$$
obtained in \cite[Theorem 4.9]{BSWZ12}, provide the formulae
\begin{align*}
W_5'(0)&=\frac{7\zeta(3)}{2\pi^2}
-\frac1{\pi^2}\int_0^1{\textstyle\sqrt{16-x^2}}
\Re{}_3F_2\biggl(\begin{matrix} \tfrac12, \, \tfrac12, \, \tfrac12 \\ \tfrac56, \, \tfrac76 \end{matrix}\biggm| \frac{(16-x^2)^3}{108x^4} \biggr)\,\d(\log^2x)
\\ \intertext{and}
W_6'(0)&=\frac{7\zeta(3)}{2\pi^2}
-\frac1{\pi^2}\int_0^2{\textstyle\sqrt{16-x^2}}
\Re{}_3F_2\biggl(\begin{matrix} \tfrac12, \, \tfrac12, \, \tfrac12 \\ \tfrac56, \, \tfrac76 \end{matrix}\biggm| \frac{(16-x^2)^3}{108x^4} \biggr)\,\d(\log^2x)
\\ &\qquad
+\frac2{\pi^3}\int_0^2{\textstyle\sqrt{16-x^2}}
\Re{}_3F_2\biggl(\begin{matrix} \tfrac12, \, \tfrac12, \, \tfrac12 \\ \tfrac56, \, \tfrac76 \end{matrix}\biggm| \frac{(16-x^2)^3}{108x^4} \biggr)
\cdot{}_3F_2\biggl(\begin{matrix} \tfrac12, \, \tfrac12, \, \tfrac12 \\ \tfrac32, \, \tfrac32 \end{matrix}\biggm| \frac{x^2}{4} \biggr)\,\d x.
\end{align*}
These \emph{single} integrals can be used to numerically confirm the conjectural evaluations of $W_5'(0)$ and $W_6'(0)$.

A similar application of Lemma~\ref{lem1}, upon decomposing a $6$-step walk into two walks with $3$ steps, yields the alternative reduction
\begin{equation}
W_6'(0)
=2\int_0^3p_3(x)\log x\biggl(\int_0^xp_3(y)\,\d y\biggr)\,\d x,
\label{W6p3}
\end{equation}
where \cite{BSWZ12}
$$
p_3(x)=\frac{2\sqrt3x}{\pi(3+x^2)}\cdot{}_2F_1\biggl(\begin{matrix} \tfrac13, \, \tfrac23 \\ 1 \end{matrix}\biggm| \frac{x^2(9-x^2)^2}{(3+x^2)^3} \biggr).
$$
We discuss this formula further in Section~\ref{s2a+}.

Finally, we mention that equation \eqref{red1} and a modular parametrisation of $p_4(x)$ (which we indicate in Section~\ref{s2b})
were independently cast in \cite{SV14} to produce a double $L$-value expression for $W_5'(0)$.

\section{Modular parametrisation of $p_3(x)$ and related formulae}
\label{s2a+}

Note that formula \eqref{W6p3} can be written as
\begin{equation*}
W_6'(0)
=\int_0^3\log x\,\d(P_3(x)^2)
=\log3-\int_0^3P_3(x)^2\,\frac{\d x}x,
\end{equation*}
featuring the cumulative density function
$$
P_3(x)=\int_0^xp_3(y)\,\d y.
$$
The related modular parametrisation of $p_3(x)$ is given by
$$
x=x(\tau)=3\frac{\eta(\tau)^2\eta(6\tau)^4}{\eta(2\tau)^4\eta(3\tau)^2}
\colon(i\infty,0)\to(0,3),
$$
so that
$$
p_3(x)=\frac{2\sqrt3}{\pi}\,\frac{\eta(2\tau)^2\eta(6\tau)^2}{\eta(\tau)\eta(3\tau)},
\quad
\d x=3\pi i\,\frac{\eta(\tau)^6\eta(3\tau)^2\eta(6\tau)^2}{\eta(2\tau)^6}\,\d\tau
$$
and
$$
P_3(x)=6i\sqrt3\int_{i\infty}^\tau\frac{\eta(\tau)^5\eta(3\tau)\eta(6\tau)^4}{\eta(2\tau)^4}\,\d\tau
$$
is the anti-derivative of a weight 3 holomorphic Eisenstein series
$$
\frac{\eta(\tau)^5\eta(3\tau)\eta(6\tau)^4}{\eta(2\tau)^4}
=E_{3,\chi_{-3}}(\tau)-8E_{3,\chi_{-3}}(2\tau),
$$
where
$$
E_{3,\chi_{-3}}(\tau)=\frac{\eta(3\tau)^9}{\eta(\tau)^3}=\sum_{m,n=1}^\infty\biggl(\frac{-3}m\biggr)n^2q^{mn},
\quad \chi_{-3}(m)=\biggl(\frac{-3}m\biggr)=\frac{e^{2\pi im/3}-e^{-2\pi im/3}}{i\sqrt3}.
$$
Though the anti-derivative $P_3(x)$,
\begin{align*}
P_3(x)
&=\frac{3\sqrt3}\pi\biggl(\sum_{m,n=1}^\infty\biggl(\frac{-3}m\biggr)\frac nm\,q^{mn}
-4\sum_{m,n=1}^\infty\biggl(\frac{-3}m\biggr)\frac nm\,q^{2mn}\biggr)
\\
&=\frac{9i}\pi\,\log\prod_{n=1}^\infty\biggl(\frac{(1-e^{2\pi i/3}q^{2n})^4(1-e^{-2\pi i/3}q^n)}{(1-e^{-2\pi i/3}q^{2n})^4(1-e^{2\pi i/3}q^n)}\biggr)^n,
\end{align*}
is not considered to be sufficiently ``natural'',
it shows up as the elliptic dilogarithm thanks to Bloch's formula; see \cite{DI07,PZ18} for the details.
Note that
$$
E_{3,\chi_{-3}}\biggl(-\frac1{3\tau}\biggr)=\frac{i\tau^3}{3\sqrt3}\,\tilde E_{3,\chi_{-3}}(\tau),
\quad
\tilde E_{3,\chi_{-3}}(\tau)
=\frac{\eta(\tau)^9}{\eta(3\tau)^3}=1-9\sum_{m,n=1}^\infty\biggl(\frac{-3}n\biggr)n^2q^{mn};
$$
and, in addition, we have
\begin{align*}
\frac1{2\pi i}\,\frac{\d x/\d\tau}{x}
=\frac12\biggl(\frac{\eta(\tau)^2\eta(3\tau)^2}{\eta(2\tau)\eta(6\tau)}\biggr)^2
&=\frac1{18}\bigl(E_{1,\chi_{-3}}(\tau)-4E_{1,\chi_{-3}}(4\tau)\bigr)^2
\\
&=\frac1{54\tau^2}\biggl(E_{1,\chi_{-3}}\biggl(-\frac1{12\tau}\biggr)-E_{1,\chi_{-3}}\biggl(-\frac1{3\tau}\biggr)\biggr)^2,
\end{align*}
where
$$
E_{1,\chi_{-3}}(\tau)=1+6\sum_{m,n=1}^\infty\biggl(\frac{-3}m\biggr)q^{mn}.
$$

\section{Modular computation for $W_5'(0)$ and $W_6'(0)$}
\label{s2b}

As (partly) shown in \cite{BSWZ12} the density $p_4(x)$ can be parameterised as follows (we make a shift of $\tau$ by half):
$$
p_4(x(\tau))=-\Re\biggl(\frac{2i(1+6\tau+12\tau^2)}\pi\,p(\tau)\biggr),
$$
where
$$
p(\tau)=\frac{\eta(2\tau)^4\eta(6\tau)^4}{\eta(\tau)\eta(3\tau)\eta(4\tau)\eta(12\tau)}
\quad\text{and}\quad
x(\tau)=\biggl(\frac{2\eta(\tau)\eta(3\tau)\eta(4\tau)\eta(12\tau)}{\eta(2\tau)^2\eta(6\tau)^2}\biggr)^3.
$$
The path for $\tau$ along the imaginary axis from 0 to $i/(2\sqrt3)$ (or from $i\infty$ to $i/(2\sqrt3)$) corresponds to $x$ ranging from $0$ to $2$,
while the path from $i/(2\sqrt3)$ to $-1/4+i/(4\sqrt3)$ along the arc centred at 0 corresponds to the real range $(2,4)$ for~$x$.
(The arc admits the parametrisation $\tau=e^{\pi i\theta}/(2\sqrt3)$, $1/2<\theta<5/6$.)
Note that $x(i/(2\sqrt{15}))=1$ and
$$
p_4(x(\tau))=\begin{cases}
-\dfrac{2i\cdot6\tau}\pi\,p(\tau), &\text{for $\tau$ on the imaginary axis}, \\[6.5pt]
-\dfrac{2i(1+6\tau+12\tau^2)}\pi\,p(\tau), &\text{for $\tau$ on the arc},
\end{cases}
$$
and
$$
-\frac{2i(1+6\tau+12\tau^2)}\pi\,p(\tau)
=\frac{2\sqrt{16-x^2}}{\pi^2x}\cdot{}_3F_2\biggl(\begin{matrix} \tfrac12, \, \tfrac12, \, \tfrac12 \\ \tfrac56, \, \tfrac76 \end{matrix}\biggm| \frac{(16-x^2)^3}{108x^4} \biggr)
$$
(this is a general form of \cite[Theorem 4.9]{BSWZ12}).
Formulas \eqref{red0}, \eqref{red1} and \eqref{red2} reduce the conjectural evaluations of $W_5'(0)$ and $W_6'(0)$ to the following ones:
$$
\frac{7\zeta(3)}{2\pi^2}+L'(f_3;-1)
\overset?=\frac{12}\pi\int_0^{1/(2\sqrt{15})}yp(iy)\log x(iy)\,\d x(iy)
$$
and
\begin{align*}
&
\frac{7\zeta(3)}{2\pi^2}+8L'(f_4;-1)
\overset?=\frac{12}\pi\int_0^{1/(2\sqrt3)}yp(iy)\log x(iy)\,\d x(iy)
\\ &\qquad
-\frac{12}{\pi^2}\int_0^{1/(2\sqrt3)}yp(iy)x(iy)
\cdot{}_3F_2\biggl(\begin{matrix} \tfrac12, \, \tfrac12, \, \tfrac12 \\ \tfrac32, \, \tfrac32 \end{matrix}\biggm| \frac{x(iy)^2}{4} \biggr)\,\d x(iy).
\end{align*}

Furthermore, note that the Atkin--Lehner involutions $w_{12}\colon\tau\mapsto-1/(12\tau)$ and $w_6\colon\allowbreak\tau\mapsto(6\tau-5)/(12\tau-6)$
act on the modular function $x(\tau)$ as follows: $x(w_{12}\tau)=x(\tau)$ and $x(w_6\tau)=-8/x(\tau)$, and we also have
$p(w_{12}\tau)=-\tau^2p(\tau)$. The point $i/(2\sqrt3)$ is fixed by~$w_{12}$.
Thus, the change of variable $y\mapsto1/(12y)$ leads to
$$
\int_0^{1/(2\sqrt3)}yp(iy)\log x(iy)\,\d x(iy)
=-\int_{1/(2\sqrt3)}^\infty yp(iy)\log x(iy)\,\d x(iy).
$$

\section{Mahler measures related to a variation of random walk}
\label{s4}

In \cite{SV14} the Mahler measures $\m(1+x_1+x_2)$ and $\m(1+x_1+x_2+x_3)$ are computed using the modular parametrisations of
$$
\sum_{n=0}^\infty W_3(2n)z^n=\sum_{n=0}^\infty\operatorname{CT}\bigl((1+x_1+x_2)(1+x_1^{-1}+x_2^{-1})\bigr)^nz^n
$$
and
$$
\sum_{n=0}^\infty W_4(2n)z^n=\sum_{n=0}^\infty\operatorname{CT}\bigl((1+x_1+x_2+x_3)(1+x_1^{-1}+x_2^{-1}+x_3^{-1})\bigr)^nz^n,
$$
where $\operatorname{CT}(L)$ denotes the constant term of a Laurent polynomial $L\in\mathbb{Z}[x_1^\pm,x_2^\pm,\ldots]$.
Note that the Picard--Fuchs linear differential equations for the two generating functions give rise to the ones for the densities
$p_3(x)$ and $p_4(x)$ together with their explicit hypergeometric and modular expressions (see \cite[eq.~(3.2) and Remark~4.10]{BSWZ12}),
though it remains unclear whether the latter information can be used to compute $W_N'(0)$ in \eqref{red0}
for $N=3,4$. This is itself an interesting question to not only assist in computing of $W_N'(0)$ for $N>4$ but also in relation with
another famous conjecture of Boyd:
\begin{equation}
\m(1+x_1+x_2+x_3+x_2x_3)\overset?=-2L'(f_2;-1)=\frac{15^2}{4\pi^4}L(f_2;3)=0.4839979734\hdots,
\label{eq:L15}
\end{equation}
where $f_2(\tau)=\eta(\tau)\eta(3\tau)\eta(5\tau)\eta(15\tau)$.

In analogy with the case of linear Mahler measures, we define
\begin{align*}
\wt W(s)
&=\iiint_{[0,1]^3}|1+e^{2\pi i\theta_1}+e^{2\pi i\theta_2}+e^{2\pi i\theta_3}+e^{2\pi i(\theta_2+\theta_3)}|^s\,\d \theta_1\,\d \theta_2\,\d \theta_3
\\
&=Z(1+x_1+x_2+x_3+x_2x_3;s)
\end{align*}
as the $s$-th moment of a random 5-step walk for which the direction of the
final step is completely determined by the two previous steps.
Then the even moments
\begin{align*}
\wt W(2n)
&=\operatorname{CT}\bigl((1+x_1+x_2+x_3+x_2x_3)(1+x_1^{-1}+x_2^{-1}+x_3^{-1}+(x_2x_3)^{-1})\bigr)^n
\\
&=\sum_{k=0}^n{\binom nk}^2{\binom{2k}k}^2
\end{align*}
satisfy a rather lengthy recurrence equation, which is equivalent to a Picard--Fuchs differential equation of order~4.
The latter splits into the tensor product of two differential equations of order~2 and, with some effort, we obtain
the following result.

\begin{theorem}
\label{th1}
We have
\begin{align*}
&
\sum_{n=0}^\infty\wt W(2n)\biggl(\frac t{(4+t)(1+4t)}\biggr)^n
\\ &\quad
=\frac{(4+t)(1+4t)}{4(1+4t+t^2)}
\,{}_2F_1\biggl(\begin{matrix} \tfrac12, \, \tfrac12 \\ 1 \end{matrix}\biggm| \frac{t(4+t)}{1+4t+t^2} \biggr)
\cdot{}_2F_1\biggl(\begin{matrix} \tfrac12, \, \tfrac12 \\ 1 \end{matrix}\biggm| \frac{t^2}{1+4t+t^2} \biggr)
\end{align*}
and, more generally,
\begin{align*}
&
\frac b{(b+t)(1+bt)}\sum_{n=0}^\infty\biggl(\frac t{(b+t)(1+bt)}\biggr)^n\sum_{k=0}^n{\binom nk}^2{\binom{2k}k}^2\biggl(\frac b4\biggr)^{2k}
\\ &\quad
={}_2F_1\biggl(\begin{matrix} \tfrac12, \, \tfrac12 \\ 1 \end{matrix}\biggm| -t(b+t) \biggr)
\cdot\frac1{(1+bt)^{1/2}}\,{}_2F_1\biggl(\begin{matrix} \tfrac12, \, \tfrac12 \\ 1 \end{matrix}\biggm| -\frac{t^2}{1+bt} \biggr)
\\ &\quad
=\frac1{1+bt+t^2}\,{}_2F_1\biggl(\begin{matrix} \tfrac12, \, \tfrac12 \\ 1 \end{matrix}\biggm| \frac{t(b+t)}{1+bt+t^2} \biggr)
\cdot{}_2F_1\biggl(\begin{matrix} \tfrac12, \, \tfrac12 \\ 1 \end{matrix}\biggm| \frac{t^2}{1+bt+t^2} \biggr).
\end{align*}
\end{theorem}

\begin{proof}
Once a factorisation of this type is written down, it is a computational routine to prove it. In other words, a principal issue
is discovering such a formula rather than proving it. Our original discovery of Theorem~\ref{th1} involved a lot of experimental
mathematics; however, we later realised that it is deducible from known formulae as follows:
\begin{align*}
\sum_{n=0}^\infty z^n\sum_{k=0}^n{\binom nk}^2{\binom{2k}k}^2x^k
&=\sum_{k=0}^\infty{\binom{2k}k}^2x^k\sum_{m=0}^\infty{\binom{k+m}k}^2z^{k+m}
\\
&=\sum_{k=0}^\infty{\binom{2k}k}^2(xz)^k\,{}_2F_1\biggl(\begin{matrix} k+1, \, k+1 \\ 1 \end{matrix}\biggm| z \biggr)
\\
&=\sum_{k=0}^\infty{\binom{2k}k}^2\frac{(xz)^k}{(1-z)^{k+1}}\,{}_2F_1\biggl(\begin{matrix} -k, \, k+1 \\ 1 \end{matrix}\biggm| -\frac z{1-z} \biggr)
\\
&=\frac1{1-z}\sum_{k=0}^\infty{\binom{2k}k}^2\biggl(\frac{xz}{1-z}\biggr)^k\cdot P_k\biggl(\frac{1+z}{1-z}\biggr),
\end{align*}
where $P_k$ denotes the $k$-th Legendre polynomial, and the latter generating function is a particular instance
of the Bailey--Brafman formula \cite{CWZ13,Zu18}.
\end{proof}

We remark that, using the general Bailey--Brafman formula and its generalisation from~\cite{WZ12}, the proof above extends to the factorisation of the two-variable generating functions
$$
\sum_{n=0}^\infty z^n\sum_{k=0}^n{\binom nk}^2\frac{(s)_k(1-s)_k}{k!^2}\,x^k
$$
as well as of
$$
\sum_{n=0}^\infty z^n\sum_k{\binom n{2k}}^2{\binom{2k}k}^2x^k
\quad\text{and}\quad
\sum_{n=0}^\infty z^n\sum_k{\binom n{3k}}^2\frac{(3k)!}{k!^3}\,x^k,
$$
and even of
$$
\sum_{n=0}^\infty z^n\sum_{k=0}^n{\binom nk}^2u_k\,x^k
$$
for an Ap\'ery-like sequence $u_0,u_1,u_2,\dots$\,.

\medskip
Furthermore, we expect that Theorem \ref{th1} can lead to a hypergeometric expression for the density function $\wt p(x)$
(piecewise analytic, with finite support on the interval $0<x<5$), which is
the inverse Mellin transform of $\wt W(s-1)$, hence to the Mahler measure evaluation
$$
\m(1+x_1+x_2+x_3+x_2x_3)
=\wt W'(0)=\int_0^\infty\wt p(x)\log x\,\d x=\int_0^5\wt p(x)\log x\,\d x.
$$

On the other hand, the reduction technique of Sections~\ref{s2+} and~\ref{s2a} suggests a different approach to computing $\wt W'(0)$,
resulting in the following hypergeometric evaluation of the Mahler measure.

\begin{theorem}
\label{th2}
We have
$$
\m(1+x_1+x_2+x_3+x_2x_3)
=-\frac1{2\pi}\int_0^1{}_2F_1\biggl(\begin{matrix} \tfrac12, \, \tfrac12 \\ 1 \end{matrix}\biggm| 1-\frac{x^2}{16} \biggr)\log x\,\d x.
$$
\end{theorem}

\begin{proof}
Define a related density $\wh p(x)$ by
\begin{align*}
\int_0^4x^s\wh p(x)\,\d x
&=\wh W(s)=\iint_{[0,1]^2}|1+e^{2\pi i\theta_2}+e^{2\pi i\theta_3}+e^{2\pi i(\theta_2+\theta_3)}|^s\,\d\theta_2\,\d\theta_3
\\
&=W_2(s)^2=\frac{\Gamma(1+s)^2}{\Gamma(1+s/2)^4}.
\end{align*}
By an application of the Mellin transform calculus, we find that, for $0<x<4$,
$$
\wh p(x)=\frac1{2\pi}\cdot{}_2F_1\biggl(\begin{matrix} \tfrac12, \, \tfrac12 \\ 1 \end{matrix}\biggm| 1-\frac{x^2}{16} \biggr).
$$
Then it follows from Lemma~\ref{lem1} that
$$
\wt W'(0)=\int_1^4\wh p(x)\log x\,\d x
=-\int_0^1\wh p(x)\log x\,\d x,
$$
where we use the evaluation
\begin{equation*}
\int_0^4\wh p(x)\log x\,\d x=\m(1+x_2+x_3+x_2x_3)=\m(1+x_2)+\m(1+x_3)=0.
\qedhere
\end{equation*}
\end{proof}

The above proof extends to the general formula
\begin{align*}
\m(1+bx_1+x_2+x_3+x_2x_3)
&=\log b\int_0^b\wh p(x)\,\d x+\int_b^4\wh p(x)\log x\,\d x
\\
&=\frac1{2\pi}\int_0^b{}_2F_1\biggl(\begin{matrix} \tfrac12, \, \tfrac12 \\ 1 \end{matrix}\biggm| 1-\frac{x^2}{16} \biggr)\,\log\frac bx\,\d x
\end{align*}
for $0<b\le4$.
A related computation
$$
\m(1+bx_1+x_2+x_3+x_2x_3)
=\log b+\frac{8}{\pi^2}\int_b^4\frac{\arccos(b/x)\,\log(x/(2\sqrt{b}))}{\sqrt{16-x^2}}\,\d x
$$
valid for $0<b\le4$ was given by J.~Wan~\cite{Wa11}; he also pointed out that
$\m(1+bx_1+x_2+x_3+x_2x_3)=\log b$ for $b>4$ follows from Jensen's formula.

\medskip
The left-hand side of another Mahler measure conjecture \cite{BLRVD03}
$$
\m((1+x_1)^2+x_2+x_3)\overset?=-L'(\tilde f_2;-1)=\frac{72}{\pi^4}\,L(\tilde f_2;3)
=0.7025655062\dots,
$$
where $\tilde f_2(\tau)=\eta(2\tau)\eta(4\tau)\eta(6\tau)\eta(12\tau)$ is a cusp form of level 24,
can be treated by a similar reduction, using that the densities for $(1+x_1)^2$ and $x_2+x_3$ are
$p_2(t^{1/2})/(2t^{1/2})$ on $[0,4]$ and $p_2(t)$ on $[0,2]$, respectively.
The final result is the elegant formula
\begin{equation}
\m((1+x_1)^2+x_2+x_3)=\frac{2G}\pi+\frac2{\pi^2}\int_0^1\arcsin(1-x)\,\arcsin x\,\frac{\d x}x,
\label{eq:L24}
\end{equation}
where $G$ is Catalan's constant,
and, with some further work, we can express the right-hand side hypergeometrically.

\begin{theorem}
\label{th3}
We have
\begin{align*}
\m((1+x_1)^2+x_2+x_3)
&=\frac{8\Gamma(\frac34)^2}{\pi^{5/2}}\,
{}_5F_4\biggl(\begin{matrix} \tfrac14, \, \tfrac14, \, \tfrac14, \, \tfrac34, \, \tfrac34 \\ \frac12, \, \tfrac54, \, \tfrac54, \, \tfrac54 \end{matrix}\biggm| \frac 14 \biggr)
\\ &\qquad
+\frac{\Gamma(\frac14)^2}{54\pi^{5/2}}\,
{}_5F_4\biggl(\begin{matrix} \tfrac34, \, \tfrac34, \, \tfrac34, \, \tfrac54, \, \tfrac54 \\ \frac32, \, \tfrac74, \, \tfrac74, \, \tfrac74 \end{matrix}\biggm| \frac 14 \biggr).
\end{align*}
\end{theorem}

\begin{proof}
Notice that, for $0<x<1$,
$$
\arcsin(1-x)=\frac\pi2-\arccos(1-x)
=\frac\pi2-\sqrt{2x}\,{}_2F_1\biggl(\begin{matrix} \tfrac12, \, \tfrac12 \\ \frac32 \end{matrix}\biggm| \frac x2 \biggr),
$$
and that, for $n>-1/2$,
$$
\int_0^1x^{n-1/2}\arcsin x\,\d x
=\frac{\sqrt\pi}{2n+1}\biggl(\sqrt\pi-\frac{\Gamma(\frac n2+\frac34)}{\Gamma(\frac n2+\frac54)}\biggr).
$$
Therefore,
\begin{align*}
\int_0^1\arcsin(1-x)\,\arcsin x\,\frac{\d x}x
&=\frac\pi2\int_0^1\arcsin x\,\frac{\d x}x
-\pi\sqrt2\sum_{n=0}^\infty\frac{(\frac12)_n^2}{n!\,(\frac32)_n(2n+1)}\,\frac1{2^n}
\\ &\qquad
+\sqrt{2\pi}\sum_{n=0}^\infty\frac{(\frac12)_n^2\Gamma(\frac n2+\frac34)}{n!\,(\frac32)_n(2n+1)\,\Gamma(\frac n2+\frac54)}\,\frac1{2^n}.
\end{align*}
From this and \eqref{eq:L24} we deduce
\begin{align*}
&
\m((1+x_1)^2+x_2+x_3)=\frac{2G}\pi+\frac{\log2}2
-\frac{2\sqrt2}\pi\,{}_3F_2\biggl(\begin{matrix} \tfrac12, \, \tfrac12, \, \tfrac12 \\ \frac32, \, \tfrac32 \end{matrix}\biggm| \frac 12 \biggr)
\\ &\quad
+\frac{8\sqrt2\,\Gamma(\frac34)}{\pi^{3/2}\Gamma(\frac14)}\,
{}_5F_4\biggl(\begin{matrix} \tfrac14, \, \tfrac14, \, \tfrac14, \, \tfrac34, \, \tfrac34 \\ \frac12, \, \tfrac54, \, \tfrac54, \, \tfrac54 \end{matrix}\biggm| \frac 14 \biggr)
+\frac{\sqrt2\,\Gamma(\frac14)}{54\pi^{3/2}\Gamma(\frac34)}\,
{}_5F_4\biggl(\begin{matrix} \tfrac34, \, \tfrac34, \, \tfrac34, \, \tfrac54, \, \tfrac54 \\ \frac32, \, \tfrac74, \, \tfrac74, \, \tfrac74 \end{matrix}\biggm| \frac 14 \biggr).
\end{align*}
It remains to use
$$
G+\frac14\pi\log2
=\sqrt2\,{}_3F_2\biggl(\begin{matrix} \tfrac12, \, \tfrac12, \, \tfrac12 \\ \frac32, \, \tfrac32 \end{matrix}\biggm|\frac 12 \biggr)
$$
(see \cite[Entry 30]{Ad}) and $\Gamma(\frac14)\Gamma(\frac34)=\pi\sqrt2$.
\end{proof}

\section{Conclusion}
\label{s5}

A goal of this final section is to highlight relevance for and links with other research and open problems.

The (hypergeometric) factorisation in Theorem~\ref{th1} and similar results outlined after its proof are part of
a general phenomenon of arithmetic differential equations of order 4. These are the first instances ``beyond modularity''
in the sense that arithmetic differential equations of order 2 and 3 are always supplied by modular parametrisation.
In order 4, we have to distinguish two particular novel situations (though our knowledge about either is imperfect and incomplete): (the Zariski closure of) the monodromy group is the orthogonal group $O_4\simeq O_{2,2}$
of dimension~6 or the symplectic group $Sp_4$ of dimension~10.
The example given in Theorem~\ref{th1} corresponds to the first (orthogonal) situation:
on the level of Lie groups, $O_{2,2}$ can be realised as the tensor product of two copies of $SL_2$ (or $GL_2$).
There is a limited amount of further examples of this type \cite{RS13,WZ12,Zu14} though we expect that all underlying Picard--Fuchs differential equations
with such monodromy can be represented as tensor products of two arithmetic differential equations of order~2.
There is a natural hypergeometric production of such orthogonal cases using Orr-type formulae (see \cite{Gu15,Wa14})
but there are plenty of other cases coming from classical work of W.\,N.~Bailey and its recent generalisations~\cite{WZ12,Zu18}.
Many such cases, mostly forecast by Sun~\cite{Su14}, are still awaiting their explicit factorisation.
Though these situations do not cover symplectic monodromy instances, they can still be viewed as an intermediate step
between classical modularity and $Sp_4$: the antisymmetric square of the latter happens to be $O_5\simeq O_{3,2}$
(see \cite{ASZ11}).

More in the direction of three-variable Mahler measure, the conjectural evaluation in \eqref{eq:L15} and Theorem~\ref{th2}
brings us to the expectation
\begin{equation}
\frac1{2\pi}\int_0^1{}_2F_1\biggl(\begin{matrix} \tfrac12, \, \tfrac12 \\ 1 \end{matrix}\biggm| 1-\frac{x^2}{16} \biggr)\log x\,\d x\overset?=2L'(f_2;-1).
\label{eq:L15-1}
\end{equation}
This one highly resembles the evaluation
\begin{equation}
\frac12\int_0^1{}_2F_1\biggl(\begin{matrix} \tfrac12, \, \tfrac12 \\ 1 \end{matrix}\biggm| \frac{x^2}{16} \biggr)\,\d x
=\frac12\cdot{}_3F_2\biggl(\begin{matrix} \tfrac12, \, \tfrac12, \, \tfrac12 \\ 1, \, \tfrac32 \end{matrix}\biggm| \frac1{16} \biggr)
=2L'(f_2;0)
\label{eq:L15-0}
\end{equation}
established in~\cite{RZ14}. The related modular parametrisation
$$
x=x(\tau)=16\biggl(\frac{\eta(\tau)\eta(4\tau)^2}{\eta(2\tau)^3}\biggr)^4
$$
corresponds to
\begin{gather*}
1-\frac{x^2}{16}=\biggl(\frac{\eta(\tau)^2\eta(4\tau)}{\eta(2\tau)^3}\biggr)^8,
\\
F\biggl(\frac{x^2}{16}\biggr)=\frac{\eta(2\tau)^{10}}{\eta(\tau)^4\eta(4\tau)^4}
\quad\text{and}\quad
F\biggl(1-\frac{x^2}{16}\biggr)=-2i\tau F\biggl(\frac{x^2}{16}\biggr),
\end{gather*}
where $F$ denotes the corresponding $_2F_1$ hypergeometric series. Note that $x$ ranges from $0$ to $4$
when $\tau$ runs from $i\infty$ to $0$ along the imaginary axis; however, the point $\tau_0=i\,0.8774376613482\dots$,
at which $x(\tau_0)=1$, is not a quadratic irrationality.
Furthermore, H.~Cohen \cite{Co18} observes another step in the ladder \eqref{eq:L15-0}, \eqref{eq:L15-1}:
\begin{align}
\frac6{\pi^2}\int_0^1{}_2F_1\biggl(\begin{matrix} \tfrac12, \, \tfrac12 \\ 1 \end{matrix}\biggm| \frac{x^2}{16} \biggr)\log^2x\,\d x
&\overset?=2L'(f_2;-2)=\frac{3\cdot15^3}{8\pi^6}\,L(f_2;4)
\label{eq:L15-2}
\\
&=1.2165632526\hdots,
\nonumber
\end{align}
though not linked to a particular Mahler measure.

The expression in Theorem~\ref{th3} is somewhat different from the one in Theorem~\ref{th2},
and resembles the hypergeometric evaluation of the $L$-value
\begin{align*}
-L'(\hat f_2;-1)
&=\frac{128}{\pi^4}L(\hat f_2;3)
\\
&=\frac{\Gamma(\frac14)^2}{6\sqrt2\pi^{5/2}}\,
{}_4F_3\biggl(\begin{matrix} 1, \, 1, \, 1, \, \frac12 \\ \frac74, \, \frac32, \, \frac32 \end{matrix}\biggm| 1 \biggr)
+\frac{4\Gamma(\frac34)^2}{\sqrt2\pi^{5/2}}\,
{}_4F_3\biggl(\begin{matrix} 1, \, 1, \, 1, \, \frac12 \\ \frac54, \, \frac32, \, \frac32 \end{matrix}\biggm| 1 \biggr)
\\ &\qquad
+\frac{\Gamma(\frac14)^2}{2\sqrt2\pi^{5/2}}\,
{}_4F_3\biggl(\begin{matrix} 1, \, 1, \, 1, \, \frac12 \\ \frac34, \, \frac32, \, \frac32 \end{matrix}\biggm| 1 \biggr),
\end{align*}
where $\hat f_2(\tau)=\eta(4\tau)^2\eta(8\tau)^2$ is a cusp form of level~32, obtained in~\cite[Theorem 3]{Zu13}.

Finally, we remark that the integral
$$
W_3'(0)=\int_0^3\log x\,\d P_3(x)
=\log3-\int_0^3P_3(x)\,\frac{\d x}x
$$
in the notation of Section~\ref{s2a+}, with $P_3(x)$ related to eta quotients, is visually linked to the following result in \cite{BZ02}
(also discussed in greater generality in \cite{ABYZ02,Ta06})
$$
\int_0^1\frac19\biggl(1-\frac{\eta(\tau)^9}{\eta(3\tau)^3}\biggr)\frac{\d q}q
=\lim_{q\to1^-}\sum_{m,n=1}^\infty\biggl(\frac{-3}n\biggr)\frac nm\,q^{mn}
=L'(\chi_{-3};-1).
$$
However, apart from the fact that the two quantities coincide we could not find a direct link between the two integrals.

\medskip
\noindent
\textbf{Acknowledgements.}
We thank H.~Cohen for supplying us with the numerical observation \eqref{eq:L15-2} whose origin remains completely mysterious to us. We also thank the referee for their valuable feedback.


\end{document}